\documentclass[reqno]{amsart}
\usepackage{latexsym, amsfonts, euscript}
\usepackage{epsfig, url}
\usepackage{graphicx,color}
\newtheorem{theorem}{Theorem}[section]
\newtheorem{proposition}[theorem]{Proposition}

\theoremstyle{remark}
\newtheorem{remark}[theorem]{Remark}
\newtheorem{remarks}[theorem]{Remarks}

\def\R{{\mathbb R}}

\begin{document}
\title{Sylvester's Minorant Criterion, Lagrange-Beltrami Identity, and Nonnegative Definiteness}
\renewcommand{\rightmark}
{\textsc\large{SYLVESTER'S CRITERION AND LAGRANGE-BELTRAMI IDENTITY}}
\author{Sudhir R. Ghorpade}
\address{Department of Mathematics, 
Indian Institute of Technology Bombay,\newline \indent
Powai, Mumbai 400076, India.}
\email{srg@math.iitb.ac.in}
\urladdr{http://www.math.iitb.ac.in/$\sim$srg/}

\author{Balmohan V. Limaye}
\address{Department of Mathematics, 
Indian Institute of Technology Bombay,\newline \indent
Powai, Mumbai 400076, India.}
\email{bvl@math.iitb.ac.in}
\urladdr{http://www.math.iitb.ac.in/$\sim$bvl/}
\subjclass[2000]{15A57, 15A15}
\keywords{Positive definite, nonnegative definite, principal minor.} 

\begin{abstract}
We consider the characterizations of positive definite as well as
nonnegative definite quadratic forms in terms of the principal
minors of the associated symmetric matrix. We briefly review some of
the known proofs, including a classical approach via the
Lagrange-Beltrami identity.  For quadratic forms in up to $3$
variables, we give an elementary and self-contained proof of
Sylvester's Criterion for positive definiteness as well as for
nonnegative definiteness. In the process, we obtain an explicit
version of Lagrange-Beltrami identity for ternary quadratic forms.
\end{abstract}

\maketitle

\section{Introduction}
\label{sec:intro}

Let $A=(a_{ij})$ be an $n\times n$ real symmetric matrix and
$$
Q(\mathbf{x})=Q(x_1, \dots, x_n):= \mathbf{x}A\mathbf{x}^T = \sum_{i=1}^n \sum_{j=1}^n a_{ij} \, x_i x_j \,
$$
be the corresponding (real) quadratic form in $n$ variables. Recall that the matrix $A$  or the form $Q$ is said to be \emph{positive definite} (resp: \emph{nonnegative definite}\footnote{Sometimes the term \emph{positive semi-definite} is used in place of \emph{nonnegative definite}. On the other hand, some books (e.g., \cite{CJ,Ho,Wi}) define a \emph{positive semi-definite} quadratic form as one which is nonnegative definite but not positive definite.}) if
$Q(\mathbf{x})>0$ (resp: $Q(\mathbf{x}) \ge 0$) for all $\mathbf{x}\in \R^n$, $\mathbf{x}\ne \mathbf{0}$.
For any matrix, a \emph{minor} is the determinant of a square submatrix. 
A minor is called a \emph{principal minor\,} if the rows and columns chosen to form the submatrix have the same indices;
further, if these indices are consecutive and start from $1$, then it is called a \emph{leading principal minor}.
Sylvester's minorant criterion 
is the well-known result that
\begin{equation}
A \mbox{ is positive definite} \Longleftrightarrow \mbox{the leading principal minors of $A$ are positive.}
\label{posdef}
\end{equation}
An analogous characterization for nonnegative definite matrices seems  relatively less well-known and conspicuous by its absence in most texts on Linear Algebra. It may be tempting to guess that $A$ is nonnegative definite if and only if all the leading principal minors of $A$ are nonnegative. In fact, some books (e.g.,   \cite[Ch. 2 \S 5]{BB} or  \cite[p. 133]{Wi}) appear to state incorrectly 
that this is true.
To see that nonnegativity of leading principal minors does not imply  nonnegative definiteness, it suffices to consider the matrix
$$
A = \left(\begin{array}{rr} 0 & \ 0  \\
0 & \ -1   \end{array} \right) \quad \mbox{ or the corresponding quadratic form } \quad  Q(x,y) = -y^2.
$$
In \cite[p.293]{Bo}, this example is given and the author also
states that for positive semi-definiteness, there is no straightforward generalization of Sylvester's minorant criterion! Nonetheless there is a simple and natural generalization  
as follows.
\begin{equation}
A \mbox{ is nonnegative definite} \Longleftrightarrow \mbox{the principal minors of $A$ are nonnegative.}
\label{nonnegdef}
\end{equation}
The aim of this article is to effectuate a greater awareness of \eqref{nonnegdef} and a related algebraic fact known as
the Lagrange-Beltrami identity. The existing proofs in the literature
of \eqref{nonnegdef} as well as \eqref{posdef}  seem rather involved.
(See Remark \ref{rem:onpfs} and the paragraph before Proposition \ref{pro:bqf}.) With this in view, we shall outline a completely self-contained and elementary proof of \eqref{posdef} and \eqref{nonnegdef} 
when $n \le 3$. There is a good reason why such a proof may be useful and interesting. As is well-known, characterizations of positive definite matrices, when applied to the Hessian matrix, play a crucial role in the local analysis of real-valued functions of several real variables. 
For example, they give rise to the so called Discriminant Test, which is a useful criterion to determine a local extremum or a saddle point.  Characterizations of nonnegative definiteness are also useful here, and more importantly, in the study of convexity and concavity of functions of several variables. (See, for example, 
\cite[\S 42]{varberg}.) Usually these topics are studied in Calculus courses before the students have an exposure to Linear Algebra and learn notions such as eigenvalues and results such as the Spectral Theorem. Also, it is common to restrict to functions of two or three variables. Thus it seems desirable to have a proof for $n\le 3$ that assumes only the definition of the determinant of a $2\times 2$ or $3\times 3$ matrix.

As indicated earlier, the pursuit of an elementary proof leads one to the Lagrange-Beltrami identity, which is yet another gem from classical linear algebra that  deserves to be better known and better understood. In Section \ref{sec2} below, we explain this identity and illustrate its use in proving  \eqref{posdef} and \eqref{nonnegdef} in the simplest case  $n=2$. We also comment on some of the existing proofs of \eqref{posdef} and \eqref{nonnegdef} in the general case. Section \ref{sec3} deals with the case $n=3$, and ends with a  number of remarks and a related question.

\section{Lagrange-Beltrami Identity and Binary Quadratic forms}
\label{sec2}

Let $A=(a_{ij})$ and $Q(\mathbf{x})$ 
be as in the Introduction. For $1\le k\le n$, let
$$
\Delta_k  := \left| \begin{array}{ccc} a_{11} & \dots & a_{1k}  \\
\vdots & & \vdots \\
a_{k1} & \dots  & a_{kk}   \end{array} \right|
$$
be the $k$th leading principal minor of $A$. 
Set $\Delta_0 := 1$.
Evidently, a natural way to prove the implication `$\Leftarrow$' in \eqref{posdef} is to show that if $\Delta_k > 0$ for $1\le k\le n$, then $Q(x)$ is a sum of squares that vanishes only when $\mathbf{x} = \mathbf{0}$.
The Lagrange-Beltrami identity does just this. It states that if
the product $\Delta_1 \cdots \Delta_{n-1}$ 
is nonzero, then
\begin{equation}
Q(\mathbf{x})=  \sum_{i=1}^n
\frac{\Delta_i}{\Delta_{i-1}} \, y_i^2, \quad \mbox{ where } \quad
y_i= x_i + \sum_{i<j\le n} b_{ij}x_j  \quad \mbox{for } i=1, \dots , n,
\label{LB}
\end{equation}
and each $b_{ij}$ is a rational function in the entries of $A$.
Notice that the equations for $y_1, \dots, y_n$ 
in terms of $x_1, \dots, x_n$ 
are in a triangular form; 
hence if $\Delta_k > 0$ for $1\le k\le n$, then
$$
Q(\mathbf{x})= \mathbf{0} \Longleftrightarrow y_i =0 \mbox{ for } i=1, \dots , n
\Longleftrightarrow  x_i =0 \mbox{ for } i=1, \dots , n.
$$

To prove the other implication `$\Rightarrow$' in \eqref{posdef}, it is customary to use induction on $n$ together with the following fact.
\begin{equation}
A \mbox{ is positive definite } \Longrightarrow \det A > 0.
\label{posdet}
\end{equation}
This fact follows readily from the following eigenvalue characterization:
\begin{equation}
A \mbox{ is positive definite } \Longleftrightarrow \mbox{ the eigenvalues of $A$ are positive.}
\label{eigenpos}
\end{equation}
In turn, \eqref{eigenpos} follows from the Spectral Theorem for real symmetric matrices. In the case of nonnegative definiteness, we have a straightforward analogue of \eqref{eigenpos}, namely,
\begin{equation}
A \mbox{ is nonnegative definite } \Longleftrightarrow \mbox{ the eigenvalues of $A$ are nonnegative.}
\label{eigennonneg}
\end{equation}
This, too, follows from the Spectral Theorem. Also, as a consequence, we have an obvious analogue of \eqref{posdet} that together with induction on $n$ will prove the implication `$\Rightarrow$' in \eqref{nonnegdef}.
However, if one is seeking an elementary proof, one should try to avoid the use of the Spectral Theorem and its consequences such as \eqref{eigenpos} and \eqref{eigennonneg}. Also, if some $\Delta_i=0$, then to prove \eqref{nonnegdef},
the Lagrange-Beltrami identity \eqref{LB} seems useless even if we clear the denominators. We will now see that at least for small values of $n$, the Lagrange-Beltrami identity is still useful if we know it explicitly and also if we know some of its {\it avatars}. Moreover, the use of Spectral Theorem can be avoided by suitable `substitution tricks'. Let us illustrate by considering the simplest case of binary quadratic forms, that is, the case of $n=2$.

\begin{proposition}
\label{pro:bqf}
Let $Q(x,y):=ax^2 + 2bxy+cy^2$ be a binary quadratic form in the variables $x$ and $y$ with coefficients $a,b,c$ in $\R$. Then
$$
Q(x,y) \mbox{ is nonnegative definite } \Longleftrightarrow a\ge 0, \ c\ge 0 \mbox{ and } ac - b^2 \ge 0.
$$
\end{proposition}

\begin{proof}
Suppose $Q(x,y)$ is nonnegative definite. Then $a  = Q(1,0) \ge 0$ and
$c=Q(0,1)\ge 0$. In case $a\ne 0$, consider
$$
Q(b, -a) = ab^2 - 2ab^2 + ca^2 = ca^2 - ab^2 = a (ac-b^2).
$$
Since $Q(b, -a)\ge 0$ and $a > 0$, we must have $ac-b^2 \ge 0$. Next, in case $a=0$
and $c\ne 0$, consider
$$
Q(c,-b) = ac^2 - 2cb^2 + cb^2 = ac^2 - cb^2 = c(ac - b^2).
$$
Since $Q(c,-b)\ge 0$ and $c>0$, we must have $ac-b^2 \ge 0$. Finally, in case
$a=0$ and $c=0$, we have
$2b = Q(1,1) \ge 0$ and $-2b = Q(1,-1)\ge 0$, which implies that $b=0$; hence, in this
case $ac-b^2 = 0$.

Conversely, suppose $a\ge 0$, $c\ge 0$ and $ac-b^2 \ge 0$. Let $\varDelta:=ac - b^2$. In case $a > 0$, the identity
$$
aQ(x,y) = a^2 x^2 + 2ab xy + acy^2=(ax+by)^2+\varDelta y^2 
$$
implies that $Q(x,y)\ge 0$ for all $(x,y)\in \R^2$. In case $c>0$, the identity
$$
cQ(x,y) = ac x^2 + 2bc xy + c^2y^2=(bx+cy)^2+\varDelta x^2 
$$
implies that $Q(s,t)\ge 0$ for all $(s,t)\in \R^2$. In case $a=c=0$, the condition
$ac-b^2 \ge 0$ implies that $b=0$, and hence $Q(s,t) = 0$ for all $(s,t)\in \R^2$.
Thus, in any case, $Q(x,y)$ is nonnegative definite.
\end{proof}

It may be noted that for $n=2$, the above proposition not only yields 
\eqref{nonnegdef}  but the arguments in the proof readily
yield \eqref{posdef} as well. In fact, proving \eqref{posdef} is simpler because one has to consider fewer cases.

\begin{remark}
\label{rem:onpfs}
A proof of \eqref{posdef} using the Lagrange-Beltrami identity appears, for example, in \cite{BB, Be, Ho}. Other proofs, as can be found, for example, in \cite{Fr, Gi, HJ}, use an inductive argument based on \eqref{eigenpos} and something like the Interlacing Theorem \cite[Thm. 7.3.9]{HJ}
or 
a version of the Courant-Fischer ``min-max theorem'' \cite[Thm. 4.2.11]{HJ}.
As we have noted already, a proof of \eqref{eigenpos}
uses the Spectral Theorem.
The Lagrange-Beltrami identity can be proved by an inductive argument
using essentially the ideas of Gaussian elimination
(cf. \cite[Ch. 5, \S 2]{Be} or \cite[\S 9.17]{Ho}). 
As for the characterization \eqref{nonnegdef} of nonnegative definiteness,  one of the implications in \eqref{nonnegdef}  appears as an exercise in \cite[p. 405]{HJ}.
A complete statement together with some related characterizations and an outline of a proof can be found in \cite[\S 9.3]{St}, while \cite[Thm. 9.4.9]{RB} gives a more detailed proof. These proofs are not difficult except that they use all those `standard theorems' that one usually discusses toward the fag end of a serious course in Linear Algebra.
\end{remark}

\section{Ternary Quadratic Forms}
\label{sec3}

The substitution tricks and an explicit version of the Lagrange-Beltrami identity together with its {\it avatars} can be used in proving \eqref{nonnegdef} for ternary quadratic forms, that is, for $n=3$, as follows. In the statement below, we have avoided the use of subscripts for the entries of $A$ or the coefficients of $Q$.
The only thing used in the proof is the corresponding result for binary quadratic forms 
(Proposition \ref{pro:bqf}). 

\begin{proposition}
\label{Prop:5.3}
Let $Q(x,y, z):= ax^2 +2bxy + 2pxz + cy^2 + 2qyz + rz^2$ be a ternary quadratic form in the variables $x$, $y$ and $z$ with
coefficients $a,b,c,p,q,r$ in $\R$. Let
$$
A:=\left(\begin{array}{lll} a & \ b & \ p \\ b & \ c & \ q \\
p & \ q & \ r   \end{array} \right) \quad {\rm and} \quad
\Delta  := \det A =  p(bq-cp) + q(bp - aq) + r (ac-b^2). 
$$
Then $Q(x,y,z)$ is nonnegative definite if and only if all the principal minors of $A$ are nonnegative, i.e.,
$$
a\ge 0, \ c\ge 0, \ r\ge 0, \ ac-b^2\ge 0, \ cr-q^2\ge 0, \ ar-p^2\ge 0
\mbox{ and } \Delta \ge 0.
$$
\end{proposition}

\begin{proof}
Suppose $Q(x,y,z)$ is nonnegative definite. Then the binary quadratic forms
$Q(x,y,0)$, $Q(x,0,z)$ and $Q(0,y,z)$ are 
nonnegative definite. Hence, by Proposition \ref{pro:bqf},
each of $a,c,r, \, ac-b^2,\,  cr-q^2$ and $ar-p^2$ is nonnegative.
Further, we observe that
$
Q(bq - cp,  \ bp - aq, \ ac - b^2) = (ac - b^2)\Delta.
$
Hence if $ac - b^2 \ne 0$, then $\Delta \ge 0$. By permuting the variables $x$, $y$ and $z$ cyclically, we obtain
$Q(cr - q^2,  \ pq - br, \ bq - cp) = (cr - q^2)\Delta$
and
$Q(pq - br,  \ ar - p^2, \ bp - aq) = (ar - p^2)\Delta$.
Hence if $cr-q^2 \ne 0$ or if $ar-p^2 \ne 0$, then $\Delta \ge 0$.
Finally, suppose $ac-b^2 = cr - q^2 = ar - p^2 = 0$. Now, if $a=0$, then we must have
$b=p=0$. Similarly, if $c=0$, then $b=q=0$, while if $r=0$, then $p=q=0$. It
follows that if $acr=0$, then $\Delta =0$. Next, suppose $acr\ne 0$.
Then $bpq\ne 0$ since $b^2 =ac$, $p^2 = ar$ and $q^2 = cr$.
Consider $Q(b, -a, z) = 2z (bp-aq) + rz^2$.
Since $Q(x,y,z)$ is nonnegative definite, we see that
$2(bp-aq) + rz \ge 0$ if $z > 0$ and $2(bp-aq) + rz \le 0$ if $z < 0$.
Taking limit as  $z \to 0^+$ and also as $z \to 0^-$,
we see that $(bp-aq)\ge 0$ and also $(bp-aq)\le 0$.
Consequently, $bp-aq=0$, i.e., $aq=bp$. Hence $abq=b^2p = acp$, and therefore,
$bq = cp$. Thus $bq-cp=0$, $bp-aq=0$ and $ac-b^2=0$. It follows that $\Delta =0$.

Conversely, suppose each of  $a,c,r, \, ac-b^2,\,  ar-p^2, \, cr-q^2$ and $\Delta$ is nonnegative. In case $a=0$, then the inequalities $ac-b^2 \ge 0$ and $ar-p^2\ge 0$ imply
that $b=0$ and $p=0$. Thus, in this case, $Q(x,y,z) = cy^2+ 2qyz + rz^2$, and this
is nonnegative definite by Proposition \ref{pro:bqf}.
Similarly, if $c=0$, then $b=q=0$,
while if $r=0$, then $p=q=0$, and in either of these cases, $Q(x,y,z)$ is nonnegative definite by Proposition \ref{pro:bqf}.
Suppose $a>0$, $c>0$
and $r>0$. If
$ac-b^2 >0$, then the identity
$$
a (ac-b^2) Q(x,y,z) = (ac-b^2)(ax+by+pz)^2 + [(ac-b^2)y+(aq-bp)z]^2 + a\Delta z^2
$$
implies that $Q(x,y,z)$ is
nonnegative definite. Similarly, if $cr-q^2 >0$, then 
$$
c(cr-q^2)Q(x,y,z) = (cr-q^2)(bx+cy+qz)^2 + [(cr-q^2)z + (cp-bq)x]^2 + c\Delta x^2
$$
implies that $Q(x,y,z)$ is nonnegative definite, whereas if $ar-p^2 >0$, then 
$$
r(ar-p^2)Q(x,y,z) = (ar-p^2)(px+qy+r z)^2 + [(ar-p^2)z+(br-pq)y]^2 + r\Delta y^2
$$
implies that $Q(x,y,z)$ is nonnegative definite.   Finally, suppose $ac-b^2=ar-p^2=cr-q^2=0$. Then $bpq\ne 0$ because
$a$, $c$ and $r$ are positive. Moreover, $b^2p^2 = (ac)(ar) = a^2(cr) = a^2q^2$.
Hence $bp = \pm aq$. On the other hand, $\Delta = 2q(bp-aq)$, and so if $bp = -aq$, then $\Delta = -4a q^2 < 0$, which is a contradiction. It follows that $bp = aq$ and as a
consequence, $aQ(x,y,z) = (ax+by+pz)^2$, which implies that $Q(x,y,z)$ is
nonnegative definite.
\end{proof}

\begin{remarks}
\label{Rem:5.4}
(i) The above proof of Proposition \ref{Prop:5.3} can be easily adapted to prove \eqref{posdef} for $n=3$.
In fact, proving \eqref{posdef} 
would be much simpler since one does not have to bother with
degenerate cases.

(ii) The identity
$
a (ac-b^2) Q(x,y,z)  = (ac-b^2)(ax+by+pz)^2 + [(ac-b^2)y+(aq-bp)z]^2 + a\Delta z^2 \
$
appearing in the proof of
Proposition \ref{Prop:5.3} may be
viewed as an explicit version of the Lagrange-Beltrami identity for $n=3$.
Indeed, when $a\ne 0$ and $ac-b^2 \ne 0$, it can be written as
$$
Q(x,y,z)  = \frac{a}{1} \left( x + \frac{b}{a} y + \frac{p}{a}z\right)^2 + \frac{ac-b^2}{a}\left(y + \frac{aq-bp}{ac-b^2} z\right)^2 + \frac{\Delta}{ac-b^2} z^2.
$$
The other two similar looking identities appearing in the
proof of Proposition \ref{Prop:5.3} may be viewed as distinct
{\it avatars} of the  Lagrange-Beltrami identity for $n=3$.
In all, there are $6$ such identities expressing $M_1M_2Q(x,y,z)$ as a linear combination of squares, where $(M_1, M_2)$ is any nested sequence of $1\times 1$ and $2\times 2$ principal minors of $A$.
These readily imply a generalization \cite[Thm. 7.2.5]{HJ} of \eqref{posdef} 
in the case $n=3$. Namely, the positivity of any nested sequence of leading principal minors 
implies positive definiteness. 

(iii) Changing $A$ to $-A$ in \eqref{posdef} and \eqref{nonnegdef}, we readily obtain characterizations of negative definite matrices as well as of nonpositive definite matrices. From the characterizations of nonnegative  definite quadratic forms and nonpositive definite quadratic  forms, we can deduce a characterization of indefinite quadratic forms, that is, of those (real) quadratic forms which take positive as well as negative values.  This can be quite useful in the study of saddle points.

(iv) For the sake of simplicity, and with a view toward applications to Calculus, we have restricted ourselves to real symmetric matrices. However, the results and proofs discussed in this article extend easily to complex hermitian matrices.

(v) It will be interesting to obtain an explicit version of the Lagrange-Beltrami identity for any $n>3$, and a self-contained `high-school algebra' proof to show the equivalence of nonnegative definiteness of any quadratic form and the nonnegativity of all the principal minors of the associated matrix.
\end{remarks}

\end{document}